\documentclass[a4paper,10pt]{article}

\usepackage[utf8x]{inputenc}
\usepackage[usenames]{color}
\usepackage{amssymb}
\usepackage{graphicx}
\usepackage{amsmath}
\usepackage{amscd}
\usepackage{amsthm}
\usepackage{amsfonts}
\usepackage{bbm}

\PassOptionsToPackage{hyphens}{url}

\usepackage{IEEEtrantools}

\usepackage[colorlinks=true,
linkcolor=webgreen,
filecolor=webbrown,
citecolor=webgreen]{hyperref}

\definecolor{webgreen}{rgb}{0,.5,0}
\definecolor{webbrown}{rgb}{.6,0,0}
\usepackage{color}
\usepackage{fullpage}
\usepackage{float}

\usepackage[parfill]{parskip}
\usepackage[active]{srcltx} 

\makeatletter
\def\thm@space@setup{%
  \thm@preskip=\parskip \thm@postskip=0pt
}
\makeatother 

\newtheorem{theorem}{Theorem}[section]
\newtheorem{proposition}[theorem]{Proposition}
\newtheorem{lemma}[theorem]{Lemma}
\newtheorem{corollary}[theorem]{Corollary}
\theoremstyle{definition}

\newcommand{\setof}[2]{\{#1:#2\}}
\newcommand{\seqnum}[1]{\href{http://oeis.org/#1}{\underline{#1}}}

\newcommand{\C}{\mathbbm{C}}

\newcommand{\R}{\mathbbm{R}}
\newcommand{\Z}{\mathbbm{Z}}

\renewcommand{\l}{\lambda}

\newcommand{\mat}[1]{\text{Mat}_{#1,#1}(\Z_{\geq 0})}
\newcommand{\mats}[3]{\text{Mat}_{#1,#2}(#3)}

\usepackage{enumerate}

\begin{document}
\begin{center}
\vskip 1cm{\LARGE\bf
Counting quasi-idempotent irreducible \\  
\vskip .11in
integral matrices
}
\vskip 1cm
\large
E.~Th\"{o}rnblad \\
Department of Mathematics \\
Uppsala University \\
Box 480\\
751 06 UPPSALA \\
Sweden \\
\href{mailto:erik.thornblad@math.uu.se}{\tt erik.thornblad@math.uu.se}\\
\vskip 1cm
\large
J.~Zimmermann \\
Department of Mathematics \\
Uppsala University \\
Box 480\\
751 06 UPPSALA \\
Sweden \\
\href{mailto:jakob.zimmermann@math.uu.se}{\tt jakob.zimmermann@math.uu.se}\\
\vskip 1cm
\end{center}

\begin{abstract}
Given any polynomial $p$ in $C[X]$, we show that the set of irreducible matrices satisfying $p(A)=0$ is finite. 
In the specific case $p(X)=X^2-nX$, we count the number of irreducible matrices in this set and analyze the arising sequences and their asymptotics. 
Such matrices turn out to be related to generalized compositions and generalized partitions.
\end{abstract}

\section{Introduction}
In this paper we study the finiteness of the set of irreducible matrices which are annihilated by a given polynomial. 
This seems to be a classical problem in the spectral theory of integral matrices; however, we have not been able to find an 
answer to this question in the literature. We answer this question by the following theorem.

{\bf Theorem A.}
{\it
 For any polynomial $p \in \C[X]$, the set of irreducible integer matrices $A$ such that $p(A) = 0$ is finite.
}

A related question was studied by Eskin, Mozes and Shah \cite{EMS96}, who showed that the set of integral matrices with a given characteristic polynomial is, 
in general, not finite. All matrices in such a set are necessarily of the same dimension, something which is not true in our setting. Obviously their result 
implies that the set of integral matrices $A$ satisfying $p(A)=0$ is infinite for some polynomials $p$ (namely, the characteristic polynomials), so we cannot 
expect Theorem A to be true for arbitrary matrices. As an example, note that the polynomial
$X-1$ annihilates the identity matrix of any dimension, so the above theorem is false for any class containing all identity matrices.

The motivation to study these questions comes from higher representation theory. 
The more general finiteness problem is motivated by a technique which is used when trying to understand 
certain $2$-representations of finitary $2$-categories, see \cite{MM1}. 
The main idea is that there exists an element whose action is given by 
a non-negative, irreducible integral matrix which has to be annihilated by a certain polynomial. With this 
information one then tries to find all possible such matrices. In small cases, this can be done by 
hand, but a question that occurs quite naturally is whether or not this always is possible, i.e.\ if 
there are always only finitely many such matrices, which is why we study this question here. For more details we refer the interested reader to 
\cite{KMMZ16, MaMa16, MM5, MZ, Zi16}.

Given that there is a finite number of irreducible integral matrices annihilated by a polynomial $p$, a natural next question is whether the cardinality of this set can be determined. 
For arbitrary polynomials, this seems difficult. However, for the polynomials $f_n=X^2-nX$, this can be done. In fact, we count these matrices matrices in two different ways. 
First, we simply count all of them, which turns out to be equivalent to counting the number of generalized compositions. Secondly, we count these matrices up to permutation of 
their basis vectors. That is, we identify two $k\times k$--matrices $A,B$ if there exists a $k\times k$--permutation matrix $P$ such that conjugation of $A$ by $P$ yields $B$. 
This second case is more relevant to the problems which motivated this paper. There we know that there exists a basis such that the action is given by irreducible, non-negative
matrices. However, the order of the basis vectors does not affect the problem; in other words we do not care about it and identify matrices which can be obtained from each other by 
permutation of basis vectors. A closed formula for the number of such matrices (in either of the two cases) is not attainable, but we determine the asymptotics of 
these sequences and show that they are in bijection to classes of generalized compositions and generalized partitions.

Our interest in the polynomial $f_n=X^2-X$ stems from the following observation. Let $A$ be a finite-dimensional 
$\C$-algebra of dimension $n$. Then $F := A \otimes_{\C} A$ is an $A$-$A$-bimodule. It acts on the category of 
$A$-$A$-bimodules from the left by taking tensor products over $A$, i.e.\ for a $A$-$A$-bimodule $M$, the action 
of $F$ is given by $F(M) = A \otimes_{\C} A \otimes_{A} M$. Now, we can observe that 
\[
 F^2 = F \circ F = A \otimes_{\C} A \otimes_{A} A \otimes_{\C} A \cong (A \otimes_{\C} A)^{\oplus n} 
 = F^{\oplus n}.
\]
Thus the action of $F$ is quasi-idempotent. Therefore, on the level of the Grothendieck group,
$F$ induces a linear transformation which corresponds to a matrix $[F]$ satisfying  $f_n([F]) = 0$, i.e.\ the matrices mentioned in the paragraph above.
These kinds of problems appear, for instance, in \cite{MM5, MZ}.

The rest of the paper is outlined as follows. In the next section we introduce most of the necessary notation and preliminary results that we will need 
throughout the paper. In Section \ref{generalCase} we give a proof of Theorem A.
In Section \ref{counting} we study the integer sequences given by the number of irreducible integral matrices satisfying 
$X^2 = nX$, both when counting all matrices and when identifying matrices which are the same after permuting their basis vectors. 
We show how these sequences are related to other known sequences, and discuss their asymptotics.
\section{Preliminaries}

\subsection{Basic definitions and notation}
All matrices in this paper are assumed to be square matrices containing only integer elements. We denote by $\mats{k}{k}{\mathbb{K}}$ the set of 
$k\times k$--matrices with elements in the set $\mathbb{K}$, and denote by $0$ the matrix (of appropriate size) of all zeros.
For any matrix $A$, we denote by $(A)_{i,j}$ the element in the $i$:th row and $j$:th column of $A$. For two matrices $A, B$ of the same size, we say that $A \leq B$ provided that $(A)_{i,j} \leq (B)_{i,j}$ for all
$i, j$. If $A \leq B$ and $A \neq B$, then we write $A<B$. This defines a partial order on $\mat{k}$.

If $(A)_{i,j}>0$ for all $i,j$, then $A$ is called \emph{positive}. If  $(A)_{i,j}\geq 0$ for all $i,j$, then $A$ is called \emph{non-negative}. We say that $A$ is \emph{primitive} if it is non-negative and there exists $k > 0$ such that $A^k$ is positive.  If $A$ is non-negative and there for each
pair $i,j$ exists some $k$ such that $(A^k)_{i,j}$ is positive, then $A$ is said to be
\emph{irreducible}.

For $f \in \C[X]$, we define the sets 
\begin{align*}
 K_f^{>0} & = \bigcup_{k>0}\setof{A \in \mats{k}{k}{\Z_{> 0}}}{f(A) = 0}\\
 K_f^{\geq 0} & = \bigcup_{k>0}\setof{A \in \mats{k}{k}{\Z_{\geq 0}}}{A \text{ is irreducible}, f(A) = 0}
\end{align*}
of all irreducible square positive (resp. non-negative) integral matrices which are 
annihilated by $f$. In particular, these sets contain $1\times 1$--matrices. As mentioned above, we will study the case of $f = X^2 - nX$ 
in more detail, and define therefore
$f_n(X) := X^2 - nX$, $K_n^{\geq 0} := K_{f_n}^{\geq 0}$ and $K_n^{>0} := K_{f_n}^{>0}$. 

In Section \ref{counting} we count the elements in $K_n^{\geq 0}$ in two different ways. 
First, we simply count all of them. Secondly, we count all matrices up to permutation 
of basis vectors, by which we mean the following. Let $A, B \in \mat{k}$ and denote by $S_k$ the symmetric 
group on $k$ elements. Then, to each $\sigma \in S_k$, we assign the (permutation) matrix $P_{\sigma}$ which
is defined by $P_{\sigma}e_i = e_{\sigma(i)}$, on the elements of the standard basis $\{e_i\} \subseteq \R^k$. 
Note that $P_{\sigma}$ is an orthogonal matrix, i.e.\ $P_{\sigma}^{-1} = P_{\sigma}^t$. We say that $A$ and 
$B$ are the same up to permutation of basis vectors, denoted $A \approx B$, if there exists $\sigma \in S_k$ such that 
$P_{\sigma}^{-1}AP_{\sigma} = B$. The set of all matrices in $K_n^{\geq 0}$ up to permutation 
of basis vectors is denoted by $\overline{K}_n^{\geq 0} := K_n^{\geq 0}\slash \approx$.

\subsection{Auxiliary results}
For the proofs of our results we will need the following statement.
\begin{theorem}
 \label{noeth} 
 The semigroup $\Z_{\geq 0}^k$ is noetherian, for every $k>0$, i.e.\ all of its ideals are finitely generated.
\end{theorem}
\begin{proof}
 We may consider $\alpha \in Z_{\geq 0}^k$ as the exponent of a monomial 
$X_1^{\alpha_1}X_2^{\alpha_2}\cdots X_k^{\alpha_k}$ in  the polynomial ring 
$\C[X_1,\dots,X_n]$. With this identification, the result follows immediately from 
\cite[Lemma III.12.3]{G07}.
\end{proof}

Now, we can identify the additive semigroup $\mat{k}$ of non-negative integral $k\times k$-matrices with the 
additive semigroup
$\Z_{\geq 0}^{k^2}$ and thus we get that every ideal in this semigroup is finitely generated. This will be 
needed in Section \ref{generalCase}. 

The second important theorem which we are going to use at different points throughout this work is 
the \emph{Perron-Frobenius theorem}, more precisely the following version of it.

\begin{theorem}
 \label{pf}
 Let $A = (a_{i,j}) \in \text{Mat}_{k,k}(\R)$ be a non-negative irreducible matrix. Then the following holds.
 \begin{enumerate}[(i)]
  \item\label{pf.1} $A$ has an eigenvalue $r_A = r > 0$, the so-called \emph{Perron-Frobenius} eigenvalue, of 
  algebraic multiplicity one and  such that $r > |\lambda|$, for any other eigenvalue $\lambda$ of $A$.
  \item\label{pf.2} The Perron-Frobenius eigenvalue $r$ satisfies
        \[
         \min_i\sum_ja_{ij} \leq r \leq \max_i\sum_ja_{ij}.
        \]
  \item\label{pf.3} If $0 \leq A < B$, then $r_A \leq r_B$. Moreover, if $B$ is irreducible, then $r_A < r_B$.
 \end{enumerate}
\end{theorem}
\begin{proof}
 A proof of the first two statements can be found in Gantmacher's book \cite[Chapter XIII, \S2]{Ga59}. More precisely, \eqref{pf.1} is Theorem $2$ and \eqref{pf.2} 
 is Remark $2$. Lastly, \eqref{pf.3} follows from \cite[2.1.5 \& 2.1.10]{BP79}.
\end{proof}

\section{Finiteness proof for irreducible matrices}
\label{generalCase}
In this section we prove that the set $K_f^{\geq0}$ is finite, for any $f \in \C[X]$. Before we can do this, we need some 
notation and a lemma.

Let $x,y \in \Z_{\geq 0}^r$. We say that $x < y$ provided that $x_i \leq y_i$, for all $i$, and  there exists 
$j$ such that $x_j < y_j$. The reason for this partial order is the fact that we are going to 
study ideals $I$ in $\mat{k} \approxeq \Z_{\geq 0}^{k^2}$. By Theorem \ref{noeth}, we know that $I$ is finitely 
generated, say by some $B_1, \dots B_r$. Then,  for every $X \in I$, we have $X = \sum_{i=1}^r c_iB_i$. So, to 
every $X$ we can assign its coefficient vector $c_X = (c_i) \in \Z_{\geq 0}^r$. Now, we want to 
compare matrices in $X, Y \in I$ and then we get that $c_X < c_Y$ implies $X < Y$,  so we can 
study these coefficient vectors instead.

\begin{lemma}
 \label{infChain}
 Let $M \subseteq \Z^r$ be an infinite set. Then there exists an infinite ascending chain in $M$ with respect to 
 $<$ as defined above.
\end{lemma}
\begin{proof}
 Note that $M$ is countably infinite, so there is an enumeration of $M = \{m^{(n)}\}$, where 
 $m^{(n)} = (m^{(n)}_1, \dots, m^{(n)}_r)$. Now, since $M$ is infinite, there exists one component in which $m^{(n)}$
 is unbounded. Without loss of generality assume it is  the first component. Pick a subsequence $n_k$ such that
 the sequence in the first component of $m^{(n_k)}$ is strictly increasing, i.e.\ $m^{(n_k)}_1 < m^{(n_{k+1})}_1$,
 for all $k$. Then there exists a subsequence $(n_{k_l})$ which is non-decreasing in the second component, i.e.\ such
 that $m^{(n_{k_l})}_2 \leq m^{(n_{k_{l+1}})}_2$ for $l$. Similarly, by taking subsequences of subsequences we get
 a subsequence $(n_p)$ of $n_k$ such that $m^{(n_p)}_1 < m^{(n_{p+1})}_1$ and $m^{(n_p)}_i \leq m^{(n_{p+1})}_i$,
 for all $p$ and all $2 \leq i \leq r$. This yields that $m^{(n_p)}$ is an infinite ascending chain in $M$.
\end{proof}

Now we are ready to prove the main result of this section which is Theorem A from the introduction.

\begin{theorem}
\label{GenPoly}
For any polynomial $f \in \C[x]$, the set $K^{\geq 0}_f$ is finite.
\end{theorem}
\begin{proof}
Let $f = \sum_{i=0}^d \alpha_ix^i$, for $\alpha_i \in \C$ and $A \in \mat{k}$ such that $f(A) = 0$.
Then the set of eigenvalues of $A$ is a subset of the zeros of $f$. Denote by $x_0$ the zero of $f$ 
with the highest absolute value, which by Theorem \ref{pf} is an upper bound for the absolute value of 
all eigenvalues of $A$.

As a first step we prove that the size of the matrix, i.e.\ $k$, is bounded.
Since $A$ is irreducible, there exists $N > 1$ such that $B := \sum_{i = 0}^N A^i$ is positive. On the other hand 
we know that $f(A) = \sum_{i=0}^d \alpha_i A^i = 0$, which implies that $A^d$ is a $\C$-linear combination of 
smaller powers of $A$. Together, we get that there exist $\gamma_i \in \C$ such that 
\[
 B = \sum_{i=0}^d \gamma_iA^i,
\]
and as all $A^i$ are non-negative, we get that $B' := \sum_{i = 0}^d A^i$ is positive. Note that the eigenvalues of $B'$
are of the form $\sum_{i=0}^d \lambda^i$, where $\lambda$ is an eigenvalue of $A$. In particular, the largest eigenvalue $\lambda_{B'}$
of $B'$ is less than or equal to $\sum_{i=0}^d |x_0|^i$. Now, we can apply The
\ref{pf}.\eqref{pf.2} to obtain
\[
 k \leq \text{min}_i \sum_j b_{ij} \leq \l_{B'} \leq \sum_{i=0}^d|x_0|^i,
\]
since all $b_{ij} \geq 1$. Thus $k$ is bounded. 

Now fix $k$ and assume that the set $Y = \setof{A \in \mat{k}}{f(A)} = 0$ is infinite.
Consider the ideal $I\subseteq \mat{k} \simeq \Z_{\geq 0}^{k^2}$ generated by $Y$. We want to use 
the fact that $\Z_{\geq 0}^{k^2}$ is noetherian to obtain a contradiction and thus prove that 
$Y$ has to be finite.

From Theorem \ref{noeth} we get that $I$ is finitely generated. Let $B_1, \dots, B_r$ be a set of 
generators of $I$. Note that $f$ does not necessarily annihilate any of the $B_i$. 
Then we can express every $A \in Y$ as a linear combination of the $B_i$, i.e.\
\[
 A = \sum_{i=1}^r c_{A,i}B_i.
\]
The set $M = \{c_A\}$ of coefficient vectors is an infinite subset of $\Z_{\geq 0}^r$ and thus
Lemma \ref{infChain} yields that there is an infinite ascending chain $c_{A_k}$ in $M$. On the other hand, 
we have already seen that this means that this is equivalent to having an infinite ascending chain
$A_k$ of matrices in $Y$, with respect to $<$. 

However, by Theorem \ref{pf} $(3)$, this yields that there is an infinite sequence 
of different Perron-Frobenius eigenvalues, a contradiction, as all eigenvalues,   
in particular, have to be zeros of $f$.
\end{proof}

\section{Counting quasi-idempotent matrices}
\label{counting}
In this section we consider quasi-idempotent matrices, which are matrices $A$ satisfying $A^2=nA$ for some $n\geq 1$. Recall that
\begin{align*}
 K_n^{\geq 0} & = \bigcup_{k>0}\setof{A \in \mats{k}{k}{\Z_{> 0}}}{A^2 = nA}\\
 K_n^{>0} & =\bigcup_{k>0}\setof{A \in \mats{k}{k}{\Z_{\geq 0}}}{A \text{ is irreducible}, A^2 = nA}
\end{align*}

Our first result shows that such matrices factorize in a natural way. 

\begin{proposition}
\label{K_nEquivalence}
 Let $A \in \mat{k}$ be irreducible. Then the following are equivalent:
 \begin{enumerate}[(i)]
  \item $A^2 - nA = 0$, i.e.\ $A \in K_n^{\geq 0}$; \label{equiv1}
  \item $A$ has rank $1$, trace $n$ and there exist $v, w \in \Z^k_{>0}$ such that $A = vw^t$; \label{equiv2}
  \item $A^2 - nA = 0$ and $A$ is positive, i.e.\ $A \in K_n^{>0}$. \label{equiv3}
 \end{enumerate}
\end{proposition}

\begin{proof}
 Clearly \eqref{equiv3} implies \eqref{equiv1}.

 To show that \eqref{equiv1} implies \eqref{equiv2}, let $A$ be an irreducible matrix satisfying $A^2 = nA$. The only possible eigenvalues of $A$ are 
 $0$ and $n$. By Theorem \ref{pf}, we have that $n$ has to be the Perron-Frobenius eigenvalue. In particular, 
 all other eigenvalues have to be zero. Thus it follows that $A$ has rank $1$ and trace $n$, and that it
 can be written as $vw^t$ for some $v, w \in \Z^k$. However, since $A$ is irreducible, it cannot have a row or 
 column consisting of $0$'s, which means that all $v_i, w_i$ are non-zero. Since $A$ is non-negative, we have $v, w \in \mathbb{Z}_{>0}^k$. 
 
 Finally, to show that \eqref{equiv2} implies \eqref{equiv3}, let $A$ be an irreducible matrix with rank $1$ and trace $n$ such that $M = vw^t$ for some $v, w \in \Z^k_{>0}$.
 Then 
 \[
   A^2 = (vw^t)(vw^t) = v(w^tv)w^t = v\left(\sum_{i=1}^k w_iv_i\right) w^t = n \cdot vw^t = nA,
 \]
 where we use that $\sum_{i=1}^k w_iv_i$ is the trace of $A$. Since $v$ and $w$ are positive, $A$ is also positive.
\end{proof}

This result shows that, if we want to count matrices in $K_n^{\geq 0}$, we may restrict our
attention to pairs $(v,w) \in \Z_{>0}^k \times \Z_{>0}^k$ such that $\sum_{i = 1}^k v_iw_i = n$. 
Such pairs may be seen as generalized compositions as introduced by Corteel and Hitczenko \cite{CH07}. 
More formally, a \emph{generalized composition} of $n$ is a generalized word $v_1^{w_1}v_2^{w_2}\dots v_k^{w_k}$ such that $v_i, w_i >0$ and 
$\sum_i v_iw_i = n$. To simplify notation, if $v=(v_1,\dots, v_k)$ and $w=(w_1,\dots, w_k)$, then we write $v^w$ to represent the above word. We let
\begin{align*}
 \mathcal{C}_n & = \setof{v^w}{v, w \in \Z^k_{>0}, \sum_{i = 1}^k v_iw_i = n}
\end{align*}
be the set of all generalized compositions and let $c_n$ denote the cardinality of this set. 

By Proposition \ref{K_nEquivalence}, each generalized composition then corresponds to a positive matrix satisfying $A^2-nA=0$. However, 
this identification is not injective. For instance, since $(2,2)\cdot (1,1)^t = (1,1) \cdot (2,2)^t$, the generalized compositions $2^12^1$ 
and $1^21^2$ correspond to the same matrix. Therefore the question becomes which generalized compositions should be identified. It turns out 
that we only need to look at those generalized compositions which have greatest common divisor $1$, where the \emph{greatest common divisor} 
of $v^w$ is defined as follows:
\[
 \gcd{(v^w)} = \begin{cases}
              \gcd{(v_1, v_2, \ldots, v_k)}, & \text{if } k > 1, \\
              v_1, & \text{if } k = 1. 
             \end{cases}
\]

For each $n\geq 1$, let 
\[
 \mathcal{D}_n = \setof{v^w \in \mathcal{C}_n}{\gcd{(v^w)} = 1},
\]
and denote by $d_n$ the cardinality of $\mathcal{D}_n$.

\begin{lemma}
 \label{IsoComp}
 The map 
 \begin{align*}
  \varphi: \mathcal{D}_n &\to K_n^{\geq 0}    \\
	   v^w &\mapsto vw^t,
 \end{align*}
 is a bijection.
\end{lemma}

\begin{proof}
 By Proposition \ref{K_nEquivalence} it follows that $\varphi$ is well-defined. To show surjectivity, let $A \in K_n^{\geq 0}$ be a $k \times k$-matrix. 
 By Proposition \ref{K_nEquivalence} we know that there exist $v, w \in \Z_{>0}^k$ such that $M = vw^t$. 
 Moreover, if $\gcd{(v^w)} = c \neq 1$, then 
 we can write $A = (\frac{1}{c}v)(cw)^t$. Setting $v' = \frac{1}{c}v$ and $w' = cw$, we get $v'^{w'} \in \mathcal{D}_n$, as the trace of $A = v'w'^t$ is $n$ and thus
 $\varphi$ is surjective. 
 
 To show injectivity, let $v^w, v'^{w'} \in \mathcal{D}_n$ and assume that $\varphi(v^w) = \varphi(v'^{w'})$. Then $vw^t = v'w'^t$, which is only possible if 
 $v = cv'$ or $cv = v'$ for some $c \in \Z$. This in turn implies that $\gcd{(v^w)} \neq 1$ or $\gcd{(v'^{w'})} \neq 1$, a contradiction. 
\end{proof}

When we count elements of $\overline{K}_n^{\geq 0}$, we identify matrices which are equal after permutation of basis vectors.
Permuting the basis vectors of a matrix $vw^t$ corresponds to applying the same permutation to the elements of the generalized 
composition $v^w=v_1^{w_1}\dots v_k^{w_k}$. Counting up to permutation of basis vector therefore means that we should consider 
generalized partitions rather than generalized compositions. A \emph{generalized partition} of $n$ is defined to be a  
generalized composition $v_1^{w_1}v_2^{w_2}\ldots v_k^{w_k}$ where we additionally 
assume that $v_1w_1 \geq v_2w_2 \geq \cdots \geq v_kw_k$. For all $n > 0$, we let
\begin{align*}
 \mathcal{P}_n & = \setof{v^w}{v, w \in \Z^k_{>0}, \sum_{i = 1}^k v_iw_i = n, v_1w_1 \geq v_2w_2 \geq \cdots \geq v_kw_k}\\
\end{align*}
be the set of all generalized partitions of $n$, and we let $p_n$ denote the cardinality of $\mathcal{P}_n$.

 Analogously to the previous case, it turns out that the set of all 
generalized partitions is slightly larger than $\overline{K}_n^{\geq 0}$ and that the correct set to consider is
\[
 \mathcal{Q}_n = \setof{v^w \in \mathcal{P}_n}{\gcd{(v^w)} = 1},
\]
i.e.\ , all generalized partitions which have greatest common divisor $1$. We denote the cardinality of $\mathcal{Q}_n$ by $q_n$.
\begin{lemma}
 \label{IsoPart}
 The map 
 \begin{align*}
  \varphi: \mathcal{Q}_n &\to \overline{K}_n^{\geq 0}    \\
	   v^w &\mapsto vw^t,
 \end{align*}
 is a bijection.
\end{lemma}

\begin{proof}
 Analogous to the proof of Lemma \ref{IsoComp}.
\end{proof}

Due to the results above, instead of considering the numbers $|K_{n}^{\geq 0}|$ and $|\overline{K}_n^{\geq 0}|$, we may consider the sequences  
$(d_n)_{n=1}^{\infty}$ (\seqnum{A280782}) and $(q_n)_{n=1}^{\infty}$ (\seqnum{A280783}), respectively. The following lemma shows that these are related via M\"{o}bius inversion to the sequences $(c_n)_{n=1}^{\infty}$ (\seqnum{A129921}) and $(p_n)_{n=1}^{\infty}$ (\seqnum{A006171}), respectively.  Recall that the M\"{o}bius function $\mu: \Z_{> 0}\to \{-1,0,1\} $ is defined on prime powers $p^k$ by 
\begin{align*}
 \mu(p^k)=\begin{cases}0 & \text{ if } k>2, \\ -1 & \text { if } k=1, \\ +1 & \text{ if } k=0.  \end{cases}
\end{align*}
and extended to non-prime powers $p_1^{k_1}p_2^{k_2}\dots p_m^{k_{m}}$ multiplicatively, i.e.\ $\mu(p_1^{k_1}p_2^{k_2}\dots p_m^{k_{m}}):=\prod_{i=1}^{m}\mu(p_i^{k_i})$.

\begin{lemma}
The sequences $(c_n)_{n=1}^{\infty}$, $(d_n)_{n=1}^{\infty}$ ,$(p_n)_{n=1}^{\infty}$ and $(q_n)_{n=1}^{\infty}$ satisfy
\begin{align*}
\begin{cases}
 c_n=\sum_{r|n}d_r \\ 
 d_n=\sum_{r|n}\mu(r)c_{n/r}
\end{cases}
\text{ and } \qquad 
\begin{cases}
     p_n=\sum_{r|n}q_r \\ 
 q_n=\sum_{r|n}\mu(r)p_{n/r}.
\end{cases}
\end{align*}

\end{lemma}

\begin{proof}
We prove the first two equalities only.
For any $n\geq 1$, let $\mathcal{C}_{n,r}$ denote the set of generalized compositions $v^w$ of $ n$ such that $\gcd(v^w)=r$. 
Note that $\{\mathcal{C}_{n,r} \ : \ r|n\}$ is a partition of $\mathcal{C}_n$. Moreover, for any $n\geq 1$ and any $r$ which divides $n$, the map 
\begin{align*}
 \mathcal{C}_{n,r} &\to \mathcal{D}_{n/r} \\
v_1^{w_1}\cdots v_k^{w_k} &\mapsto (v_1/r)^{w_1}\cdots (v_k/r)^{w_k}
\end{align*}
is easily seen to be a bijection. This proves the first equality. The second equality follows from the M\"{o}bius inversion formula (cf. \ \cite{HW60})
The proof for the sequences $(p_n)_{n=1}^{\infty}$ and $(q_n)_{n=1}^{\infty}$ is done in the same spirit.
\end{proof}

\subsection{Analysis of the asymptotics}
\label{asymptotics}
As mentioned, the sequence $(c_n)_{n=1}^{\infty}$ was studied by Corteel and Hitczenko \cite{CH07}. 
Using standard methods from analytic combinatorics, they determined the asymptotic growth rate of $(c_n)_{n=1}^{\infty}$. As we show next, 
the sequence $(d_n)_{n=1}^{\infty}$ grows asymptotically at the same rate as $(c_n)_{n=1}^{\infty}$. As a shorthand we write $a_n \sim b_n$, if
\[
 \lim_{n \to \infty} \frac{a_n}{b_n} = 1.
\]

\begin{proposition}
\label{lem:dasymp}
 As $n\to \infty$,
\begin{align}
 d_n \sim \frac{1}{\rho\sigma'(\rho)}\rho^{-n} \label{eq:dasymp}
\end{align}
where $\sigma(z)=\sum_{n\geq 1}\frac{z^n}{1-z^n}$ and $\rho$ is the unique real root of $\sigma(z)=1$ in $[0,1]$.
\end{proposition}

\begin{proof}
As shown by Corteel and Hitczenko \cite{CH07}, the asymptotics in \eqref{eq:dasymp} hold also for $(c_n)_{n=1}^{\infty}$. It suffices 
therefore to show that $c_n\sim d_n$. Note that
\begin{align*}
c_n-\log_2(n)c_{n/2} \leq \sum_{r|n}\mu(r)c_{n/r} = d_n \leq c_n,
\end{align*}
i.e.
\begin{align*}
 1-\frac{\log_2 (n) c_{n/2}}{c_n} \leq \frac{d_n}{c_n}\leq 1.
\end{align*} 
Since $(c_n)_{n=1}^{\infty}$ grows exponentially, the left hand side tends to $1$ as $n\to \infty$. This completes the proof.
\end{proof}
Approximately we have $\rho=0.406148005001\dots $, and so $  d_n \sim (0.481225\dots )(2.462156\dots )^n$.

To arrive at an asymptotic formula for $(q_n)_{n=1}^{\infty}$, we choose to analyze the more tractable $(p_n)_{n=1}^{\infty}$.
However, the asymptotics of $(p_n)_{n=1}^{\infty}$ do, to our knowledge, not exist in the literature, though they have been discussed 
on \url{mathoverflow.net} \cite{Ma}, with user lucia outlining the correct analysis. For completeness, we outline a version 
of lucia's argument in Section \ref{pnassymptotics} (and correct some incorrect terms in that answer).
There it is shown that $(p_n)_{n=1}^{\infty}$ grows superlinearly which implies that its M\"{o}bius inversion $(q_n)_{n=1}^{\infty}$ 
must grow asymptotically at the same rate, i.e.\ $q_n\sim p_n$. The proof of this is identical to that of Lemma \ref{lem:dasymp}. 
\begin{proposition}
Let $S_m$ be as in $\eqref{eq:S_m}$ below. As $n\to \infty$, 
\begin{align*}
 q_n \sim \frac{1}{\sqrt{2\pi S_1(\omega)}}\exp\left(S_{-1}(\omega)+\frac{n}{\omega} \right),
\end{align*}
where $\omega=\omega(n)\in (0,1)$ is the unique solution to $S_0(\omega)=n$.
\end{proposition}

\subsubsection{Asymptotics of generalized partitions}
\label{pnassymptotics}

We turn to the asymptotics of the sequence $(p_n)_{n=1}^{\infty}$. The generating function for generalized partitions is 
\begin{align*}
 P(z)=\sum_{n\geq 1}p_n z^n=\prod_{n=1}^{\infty}(1-z^n)^{-d(n)},
\end{align*}
where $d(n)$ is the number of divisors of $n$. Such a generating function is precisely of the form covered in Meinardus theorem \cite{FR09};
however, the corresponding Dirichlet series $\sum_{n=1}^{\infty} d(n)n^{-s} \zeta^2(s)$ has a \emph{double pole} at $s=1$, implying that Meinardus's theorem cannot be applied in this case. Results by Brigham 
\cite{B50} imply that $\log p_n\sim \pi \sqrt{\frac{n \log n}{3}}$. However, this does not imply precise asymptotics for $(p_n)_{n=1}^{\infty}$. Our argument is via the saddle-point method; in particular
we make use of more general results, which only require us to verify that $P$ satisfies certain 
conditions.

For our analysis, the series
\begin{align*}
S_m(\omega)
&:=\sum_{k \geq 1}\sum_{\ell | k} \ell d(\ell)k^{m} e^{-k/\omega}
\end{align*}
for $m\geq -1$ are important. Note that each $S_m(\omega)$ is absolutely convergent for any $\omega \in (0,\infty)$. Series of this form are amenable to analysis using Mellin transforms. 
By using convolution properties and the Mellin inversion theorem, one finds that 
\begin{align*}
S_m(\omega) = \frac{1}{2\pi i }\int_{m+3-i\infty}^{m+3+i\infty} \omega^s \Gamma(s)\zeta(s-m) \zeta(s-m-1)^2 ds.
\end{align*}
The function $\Gamma(s)$ has simple poles at $0,-1,-2,\dots $ and $\zeta(s)$ has a simple pole at $1$ and trivial zeros at $-2,-4,-6,\dots $ (which will cancel most or all poles of 
$\Gamma(s)$). Consider the contour integral of the same integrand over the rectangular contour with corners $-A-Ri$, $-A+Ri$, $m+3+Ri$, $m+3-Ri$, where $A > 1$ can be chosen arbitrarily.
This contour encloses all poles and can be evaluated using the Cauchy residue theorem. Now let $R\to \infty$ to obtain that $S_m(\omega)$
can be approximated well by the sum of the residues. We obtain, for any $A > 1$,
\begin{align}
\frac{S_m(\omega)}{(m+1)!} =  
\begin{cases}
 \zeta(2)\omega\left( \log \omega  + \gamma + \frac{\zeta'(2)}{\zeta(2)} \right) + \frac{\log \omega}{4} + \frac{\log 2\pi}{4} - \frac{1}{288 \omega} + O(\omega^{-A}) & \text{ if } m=-1, \\
 \zeta(2)\omega^2\left( \log \omega + 1 + \gamma + \frac{\zeta'(2)}{\zeta(2)} \right) + \frac{\omega}{4}+ \frac{1}{288} + O(\omega^{-A}) & \text{ if } m=0, \\
  \zeta(2)\omega^{m+2}\left(\log \omega + H(m+1) + \gamma  + \frac{\zeta'(2)}{\zeta(2)}\right) +  \frac{\omega^{m+1}}{4(m+1)} + O(\omega^{-A}) &   \text{ if } m\geq 1,
 \end{cases} \label{eq:S_m}
\end{align}
where $H(m):=\sum_{k=1}^{m} k^{-1}$, and $f(\omega)=O(g(\omega))$ means that $\limsup_{\omega \to \infty} f(\omega)/g(\omega) < \infty$.
For $|\theta| < \pi$ we have the expansion 
\begin{align*}
\log P(e^{i\theta-1/\omega})-\log P(e^{-1/\omega}) 
&= \sum_{k\geq 1}\sum_{\ell | k}\ell d(\ell)e^{-k/\omega} \left(\frac{e^{ik\theta}-1}{k} \right) \\
& = \sum_{k\geq 1}\sum_{\ell | k}\ell d(\ell)e^{-k/\omega} \sum_{m\geq 1}\frac{(i\theta)^m k^{m-1}}{m!} \\
& = \sum_{m\geq 1} \frac{(i \theta)^m}{m!}\sum_{k\geq 1}\sum_{\ell | k}\ell d(\ell)k^{m-1}e^{-k/\omega} \\
& = \sum_{m\geq 1} \frac{(i\theta)^m}{m!}S_{m-1}(\omega).
\end{align*}
Note also that $\log P(e^{-1/\omega}) =S_{-1}(\omega)$. 

The asymptotic growth rate of the sequence $(p_n)_{n=1}^{\infty}$ can be expressed in terms of $n$ and the functions $S_{-1},S_0$ and $S_1$ as follows.

\begin{theorem}\label{thm:pasymptotics}
As $n\to \infty$,
\begin{align*}
p_n\sim \frac{1}{\sqrt{2\pi S_1(\omega)}} \exp \left(S_{-1}(\omega)+\frac{n}{\omega} \right)
\end{align*}
where $\omega=\omega(n)$ is the unique solution in $(0,\infty)$ to $S_0(\omega)=n$.
\end{theorem}

\begin{proof}
The result can be deduced from the more general result \cite[Theorem VIII.4]{FR09}, provided that the function $P$ is \emph{H}--admissible, which in 
this case is equivalent to verifying that the following conditions are satisfied.
\begin{enumerate}[(i)]
\item $S_{0}(\omega)\to \infty$ and $S_1(\omega)\to \infty$ as $\omega \to \infty$.
\item There exists a function $\theta_0:(0,\infty)\to (0,\pi)$ such that 
\begin{align*}
\sum_{m\geq 3} \frac{(i\theta)^m}{m!} S_{m-1}(\omega)\to 0
\end{align*}
as $\omega \to \infty$, uniformly in $0<|\theta|<\theta_0(\omega)$.
\item For the same function $\theta_0$,
\begin{align*}
\Re\left(\frac{1}{2}\log S_1(\omega) + \sum_{m\geq 1}\frac{(i\theta)^m}{m!}S_{m-1}(\omega)\right) \to -\infty
\end{align*}
as $\omega\to \infty$, uniformly in $\theta_0(\omega)<|\theta|<\pi$. 
\end{enumerate}
Using \eqref{eq:S_m}, the first condition is obviously true, and the last two can be readily checked to hold true for $\theta_0(\omega)=\omega^{-a}$, for any $4/3<a<3/2$. Further details are omitted.
\end{proof}

\begin{corollary}
As $n\to \infty$,
\begin{align*}
p_n \sim \frac{1}{\omega^{1/4}(2n)^{1/2}} \exp\left(\frac{2n}{\omega}-\omega\left(\zeta(2)-\frac{1}{4}\right)-\frac{1}{288} \right)
\end{align*}
where $\omega=\omega(n)$ is the unique solution in $(0,\infty)$ to $S_0(\omega)=n$.
Furthermore,
\begin{align*}
\omega\sim \frac{2}{\pi} \sqrt{\frac{3n}{ \log n}}
\end{align*}
as $n\to \infty$, which implies
\begin{align*}
\log p_n \sim \pi \sqrt{\frac{n \log n}{3}}.
\end{align*}
\end{corollary}

\section{Acknowledgments.}  
The second author wants to thank his supervisor Volodymyr Mazorchuk for proposing 
the problem, many very helpful discussions and for double-checking the counting. Moreover, he would like to thank
Konstantinos Tsougkas for discussions concerning the proof of Lemma \ref{infChain}, Andrea
Pasquali for discussions regarding Section \ref{generalCase} and Daniel Tubbenhauer for discussions about Theorem \ref{GenPoly}.

\bigskip
\hrule
\bigskip

\noindent 2010 {\it Mathematics Subject Classification}: Primary 15B36; Secondary 05A15, 05A16.

\noindent \emph{Keywords: } generalized compositions, generalized partitions, quasi-idempotent matrices, positive integral matrices.

\bigskip
\hrule
\bigskip

\noindent (Concerned with sequences
\seqnum{A006171}
\seqnum{A129921},
\seqnum{A280782}, and
\seqnum{A280783}.)

\bigskip
\hrule
\bigskip

\end{document}